\documentclass[11pt, a4paper]{amsart}
\usepackage[T1]{fontenc}
\usepackage{amssymb,bbm}
\usepackage{amsthm}
\usepackage[utf8]{inputenc}
\usepackage{enumerate}
\usepackage{graphicx,color}
\usepackage{float}
\usepackage{tikz}
\usetikzlibrary{fit,shapes.geometric,arrows}

\usepackage[yyyymmdd,hhmmss]{datetime}

\newcommand{\N}{\ensuremath {\mathbb{N}} }

\newcommand{\G}{\ensuremath {\mathcal{G}}}
\newcommand{\Go}{\ensuremath {\mathcal{G}^{(0)}}}
\newcommand{\Gs}{\ensuremath {\mathcal{G}^{(2)}}}

\newcommand{\Ga}{\ensuremath {\mathcal{G}^{a}}}
\newcommand{\Gx}{\ensuremath {G \ltimes_\theta X}}
\newcommand{\Lc}{\ensuremath {\mathcal{L}_c(X)}}

\newcommand{\Lcs}{\ensuremath {\mathcal{L}_c(X_s)}}
\newcommand{\Lco}{\ensuremath {\mathcal{L}_c(\Go)}}

\newcommand{\Lg}{\ensuremath {\mathcal{L}_c(X)\rtimes_{\alpha} G}}
\newcommand{\Lgo}{\ensuremath {\mathcal{L}_c(\Go)\rtimes_{\alpha} \Ga}}

\newcommand{\Skew}{\ensuremath {\mathcal{A} \rtimes_\pi S}}

\newcommand{\an}[1]{``#1''} 

\DeclareMathOperator{\supp}{supp}

\theoremstyle{plain}
\newtheorem{theorem}{Theorem}[section]
\newtheorem{lemma}[theorem]{Lemma}
\newtheorem{prop}[theorem]{Proposition}
\newtheorem{corollary}[theorem]{Corollary}
\newtheorem{definition}[theorem]{Definition}

\newtheorem{remark}[theorem]{Remark}

\begin{document}

\author{Viviane Maria Beuter and Daniel Gonçalves} \thanks{The second author is partially supported by CNPq.}

\title{The interplay between Steinberg algebras and partial skew rings}

\date{\today,\, \currenttime }

\begin{abstract}
We study the interplay between Steinberg algebras and partial skew rings: For a partial action of a group in a Hausdorff, locally compact, totally disconnected topological space, we realize the associated partial skew group ring as a Steinberg algebra (over the transformation groupoid attached to the partial action). We then apply this realization to characterize diagonal preserving isomorphisms of partial skew group rings, over commutative algebras, in terms of continuous orbit equivalence of the associated partial actions. Finally, we show that any Steinberg algebra, associated to a Hausdorff ample groupoid, can be seen as a partial skew inverse semigroup ring.\\
\end{abstract}

\maketitle



\pagestyle{headings}
\markright{STEINBERG ALGEBRAS AND PARTIAL SKEW RINGS}
\markleft{VIVIANE M. BEUTER AND DANIEL GONÇALVES}

\section{Introduction}

The notion of crossed product by a partial action has its origin in the operator algebra context, more precisely with the work of Exel (see \cite{Exel1}), on crossed products by a partial automorphism. Over the years the theory of partial actions has proved to be an important tool in the study of C*-algebras, and in particular of C*-algebras associated to  dynamical systems and combinatorial objects, see for example \cite{E1, E2, EL, Exel, GRU, GRU1} for a small glimpse of the theory developed. 

In 2005, Dokuchaev and Exel (see \cite{Dokuchaev1}) started the study of partial skew group rings (an “algebraisation” of the operator theory concept of partial crossed products). Following the steps of the operator theory counter-part, partial skew rings are becoming an important tool in the study of algebras, in particular algebras associated with combinatorial objects. For example, Leavitt path algebras were realized as partial skew group rings in \cite{Goncalves}, and simplicity and chain conditions for partial skew group rings were studied in \cite{GoncalvesCMB, GR} and \cite{OinertChain}, respectively. Very recently, see \cite{Dokuchaev}, Dokuchaev and Exel studied the ideal structure of the partial skew group ring $\Lg,$ where $G$ is a discrete group, $X$ is a locally compact, totally disconnected space $X$, and $\Lc$ is the algebra consisting of all locally constant, compactly supported functions on $X$ taking values in a given field $K.$  


Steinberg algebras were independently introduced by Steinberg in \cite{Steinberg} and by Clark et al. in \cite{Clark1}. 
They are the “algebraisation” of Renault's C*-algebras of groupoids. Even more than partial skew group rings, the development of the theory of Steinberg algebras has attracted a lot of attention lately. In particular, Steinberg algebras include the Kumjian–Pask algebras of higher-rank graphs introduced in \cite{Aranda} (which in turn include Leavitt path algebras). See \cite{Clark}, \cite{Clark3} and \cite{Steinberg2} for a few examples of the development of the theory. 

It is our goal in this paper to link the theory of partial skew rings with the theory of Steinberg algebras, in the same way that the theory of partial crossed products is linked to Renault's theory of groupoid C*-algebras. In particular we provide an “algebraisation” of the result of Abadie, see \cite{Abadie}, that shows that any partial crossed product, associated to a partial action on a topological space, can be seen as a groupoid C*-algebra. The algebraic version of this theorem permits us to join results of Li (see \cite{Xin}), about continous orbit equivalence of partial actions on topological spaces, and results of Carlsen and Rout (see \cite{James}), about 
diagonal-preserving isomorphism between Steinberg algebras, to present results regarding diagonal preserving isomorphisms of partial skew group ring over commutative algebras. 

To complete the interplay between Steinberg algebras and partial skew rings, we show an “algebraisation” of [Theorem~3.3.1]\cite{Paterson} and [Theorem 8.1]\cite{Quigg}: Any Steinberg algebra (associated to a Hausdorff ample groupoid) can be seen as a partial skew inverse semigroup ring (partial skew inverse semigroup rings generalize partial skew group rings and were defined by Exel and Vieira in \cite{Vieira}, based on work by Nándor Sieben, see \cite{Sieben}. It is interesting to point out that the definition of a partial skew inverse semigroup ring involves a quotient by a certain ideal, which is motivated by the C*-algebraic definition of crossed products by inverse semigroups. For the reader not acquainted with the work of Exel, Paterson, Sieben, et al., on the C*-algebraic level, this quotient might seem artificial, and maybe even unnecessary. Our characterization of Steinberg algebras as partial skew inverse semigroup rings is further evidence that the quotient is necessary in the definition of partial skew inverse semigroup rings.




The organization of this paper is the following. In section 2 we set up the notation and conventions that we use through out the paper. Next, in section 3, we show that the Steinberg algebra of the transformation groupoid associated to a partial action of a discrete group $G$ on a Hausdorff, locally compact, totally disconected space $X$, is isomorphic to $\Lg.$ In section 4, we apply the isomorphism of Section~3 to characterize diagonal preserving isomorphisms of partial skew group rings over commutative algebras.  In section 5, we realize the Steinberg algebra associated to a Hausdorff ample groupoid as a partial skew inverse semigroup ring.

\section{Preliminaries}

\subsection{Inverse semigroups}
An \emph{inverse semigroup} is a nonempty set $S$ equipped with a binary associative operation (i.e., $S$ is a semigroup) such that, for each element $s \in  S,$ there exists a unique element $t \in S$ such that
$$ sts=s \,\,\,\, \mbox{and} \,\,\,\, tst=t.$$
We set $t = s^*.$ Notice that every group is an inverse semigroup with $s^*=s^{-1}.$ 

Given an inverse semigroup $S,$ one may prove that the collection of idempotent elements in $S,$ namely
$$ E(S) = \{ s \in S \, : \, s^2=s \},$$
is a commutative subsemigroup. It is immediate that if $e \in E(S)$  then $e^*=e^{-1}$, and hence $e$ can be thought of as a "projection". Under the partial order relation defined in E(S) by
$$e \leq f  \,\,\,\, \Leftrightarrow  \,\,\,\,  ef = e  \,\,\,\, \forall \, e, f \in  E(S),$$
$E(S)$ becomes a semilattice, meaning that for every $e$ and $f$ in $E(S),$ there exists a largest element which is smaller that $e$ and $f,$ namely $ef.$

There is also a natural partial order relation defined on $S$ itself by
\begin{equation}\label{semigrouporder}
s \leq t  \,\,\, \Leftrightarrow \,\,\, s= ts^*s \,\,\, \Leftrightarrow \,\,\, s=ss^*t \,\,\,\, \forall \, s, t \in S.
\end{equation}
This is compatible with the multiplication operation in the sense that
$$ s \leq t \,\,\, \mbox{and} \,\,\, u \leq v \,\,\, \Rightarrow \,\,\, su \leq tv. $$

Inverse semigroups are most easily thought of in terms of partial symmetry maps. Indeed, if $X$ is a set, then the set $I(X)$ of all partial symmetry maps on $X$ (i.e. the set of bijetive maps between subsets of $X$) is an inverse semigroup in the natural way. The product of two such maps is just the composition of the two, defined wherever it makes sense, and the
involution operation is given by inversion. The classical Wagner-Preston Theorem, asserts that any inverse semigroup can be realized as an inverse subsemigroup of some $I(X).$

\subsection{Groupoids}
A \emph{groupoid} consists of sets $\G$ and $\Gs  \subseteq \G \times \G$ such that there is a map 
$$ \Gs \ni  (b, c) \mapsto bc \in  \G, $$
and an involution 
$$ \G \ni b \mapsto b^{-1} \in \G, $$
such that the following conditions hold:
\begin{enumerate}[{\rm (a)}]
\item if $(b, c)$ and $(c, d)$ are in $\Gs$ then so are  $(bc, d)$ and  $(b,cd)$, and the equation  $(bc)d = b(cd)$ holds; 
\item for all $b \in \G$ the pair $(b, b^{-1}) \in \Gs$ and, if $(b, c) \in \Gs$ then $b^{-1}(bc) = c$ and $(bc)c^{-1}=b.$
\end{enumerate}

Associated to a groupoid there are two maps, called \emph{range} and \emph{source}, which are defined from $\G$ to itself by $r(b) = bb^{-1}$ and $s(b) = bb^{-1}$, respectively. We call the common image of $r$ and $s$ the \emph{unit space} of $\G$ and denote it by $\Go$. Note that $(b,c) \in \Gs$ if, and only if, $s(b)=r(c).$

For $B, C \subseteq \G,$ we define
\begin{equation}\label{bisectionproduct}
 BC = \{bc \in \G \, : \,  b \in B, c \in C \ \mbox{and} \ s(b)=r(c) \},
\end{equation}
and
\begin{equation}\label{bisectioninverse}
B^{-1}= \{ b^{-1} \, : \, b \in B \}.
\end{equation}

We call $\G$ a \emph{topological groupoid} if $\G$ is endowed with a topology such that composition and inversion are continuous. 

An \emph{isomorphism} of topological groupoids $\phi:\G \to \mathcal{H}$ is a homeomorphism from $\G$ to $\mathcal{H}$ that carries units to units, preserves the range and source maps, and satisfies $\phi(ab) = \phi(a)\phi(b)$ for all composable pairs $(a,b) \in \G^2$. The uniqueness of inverses implies that $\phi(a^{-1})=\phi(a)^{-1}$ for $a \in \G$.

A topological groupoid $\G$ is said to be \emph{\'{e}tale} if $\Go$ is locally compact and Hausdorff in the relative topology, and its range map is a local homeomorphism from $\G$ to $\Go$ (the source map will consequently share this property). 

A \emph{bisection} in $\G$ is an open subset $B \subseteq \G$ such that the restrictions of $r$ and $s$ to $B$ are both injective.

It is easy to see that the topology of an \'{e}tale groupoid admits a basis formed by  bisections. Also, for \'{e}tale groupoids, $\Go$ is an open bisection in $\G.$ If, in addition, $\G$ is Hausdorff, then $\Go$ is also closed in $\G.$

An \'{e}tale groupoid $\G$ is said to be \emph{ample} if it has a basis of compact open bisections. Using the notation of \cite{Paterson}, we denote by $\Ga$ the set of all the compact bisections of $\G.$ 

Let $\G$ be an \'{e}tale groupoid. If $\Go$ is totally disconnected, by \cite[Proposition~4.1]{Exel2}, $\G$ is ample. Note that the converse is true if we add the hypothesis that $\G$ is Hausdorff.



\begin{remark}
In this paper we only consider ample Hausdorff groupoids.
\end{remark} 

Let $\G$ be a Hausdorff ample groupoid.  An interesting example of an inverse semigroup, which we will use in this paper, is the set $\Ga$ consisting of all the compact bisections of $\G.$ The operations are defined as in (\ref{bisectionproduct}) and (\ref{bisectioninverse}). More precisely, we have the following proposition. 

\begin{prop}\label{bisectsetsemigroup}
Let $\G$  be a Hausdorff ample groupoid and let $ B, C, D \in \Ga.$ Then, 
\begin{enumerate}[{\rm (i)}]
\item $B^{-1}$ is a compact bisection;
\item $BC$ is a compact bisection (possibly empty);
\item $(BC)D=B(CD);$ 
\item $BB^{-1}B=B$ and $B^{-1}BB^{-1}=B^{-1};$
\item if there exist a compact bisection $C$ such that  $BCB=B$ and $CBC=C,$ then $C=B^{-1}.$
\end{enumerate} 
\end{prop}

\begin{proof} 
See \cite[Proposition~2.2.3]{Paterson}.
\end{proof}

\begin{remark} Notice that in the case above $A^*=A^{-1}$. Furthermore, if $\Go$ is compact then the inverse semigroup $\Ga$ has a unit.   
\end{remark}

For each unit $u \in \Go$ we denote $\G_u = \{ b \in \G \ : \ s(b)=u\}, \G^u:= \{ b \in \G \ : \ r(b)=u \}$ and $\G_u^u =\G_u \cap \G^u =\{b \in \G \ : \ r(b)=s(b)=u \}.$  The \emph{isotropy group} at a unit $u$ of $\G$ is the group $\G_u^u$. We say that $u$ has \emph{trivial isotropy} if $\G_u^u= \{u\}.$ A topological groupoid $\G$ is called \emph{topologically principal} if the set of points in $\Go$ with trivial isotropy is dense in $\Go.$




\subsection{Steinberg algebras}
Suppose that $\G$ is a Hausdorff, ample groupoid and $R$ is a commutative ring with unit. Let $A_R(\G)$ be the set of all functions from $\G$ to $R$ that are locally constant and have compact support. Notice that if $f \in A_R(\G)$ then
$\supp(f)=\{b \in G \ : \ f(b)\neq 0\}$ is clopen. Furthermore, if $f: \G \rightarrow R$ is locally constant then it is continuous with respect to any topology on $R.$

We define addition in $A_R(\G)$ pointwise and, for $f, g \in A_R(\G)$, convolution is defined by
\begin{equation}
(f * g)(b) = \sum_{r(c)=r(b)} f(c)g(c^{-1}b) = \sum_{cd=b}f(c)g(d)
\end{equation}
With this structure, $A_R(\G)$ is an associative algebra called the \emph{Steinberg Algebra} associated to $\G.$

Every element of $A_R(\G)$ is a linear combination of characteristic functions of pairwise disjoint compact bisections (see  \cite[Proposition~4.3]{Steinberg}). It is interesting to notice that $1_B1_C = 1_{BC},$ whenever $B$ and $C$ are compact open bisections in $\G$ (see \cite[Proposition~3.5(3)]{Steinberg}).

By \cite[Lemma~2.6]{Clark2}, the family of all idempotent elements of $A_R(\Go)$ is a set of local units for $A_R(\G).$ Moreover by \cite[Proposition~3.12.]{Steinberg}, $A_R(\G)$ has an unity if, and only if, $\Go$ is compact.

We write $D_R(\G):= A_R(\Go)$ for the commutative algebra of locally constant, compactly supported, functions from $\Go$ to $R$ with pointwise operations. Since $\Go$ is clopen there is an embedding 
\begin{equation}\label{embeddingdiagonal}
i : A_R(\Go)\rightarrow A_R(\G)
\end{equation}
such that $i(f)|_{\G \setminus \Go} =0.$ With this embedding we regard $D_R(\G)$ as a commutative subalgebra of $A_R(\G)$ and call it the \emph{diagonal} of $A_R(\G).$ When $R$ is an integral domain, we have that $D_R(\G)= span_R\{1_U \ : \ U \subseteq \Go \, \mbox{is compact open}\}.$ For ample Hausdorff groupoids $\G_1$ and $\G_2,$ we say that an isomorphism $\phi: A_R(\G_1) \to A_R(\G_2)$ is \emph{diagonal-preserving} if $\phi(D_R(\G_1))=D_R(\G_2).$

\subsection{Partial actions of inverse semigroups}\label{PAISG}

Let $S$ be an inverse semigroup and let $X$ be a set. A \emph{partial action} of $S$ on $X$ is a partial homomorphism  $\pi : S \rightarrow I(X),$ where $I(X)$ denotes the symmetric inverse semigroup of $X$ (i.e., the inverse semigroup of all partial bijections on $X$). In other terms, $\pi : S \rightarrow I(X)$ is a partial action if, for each $s, t \in S$ we have that 
\begin{center}
$ \pi(s)\pi(t)\pi(t^*)  = \pi(st)\pi(t^*), \,\,\,\, \pi(s^*)\pi(s)\pi(t)=\pi(s^*)\pi(st)$
$ \mbox{and}  \,\,\,\, \pi(s)\pi(s^*)\pi(s)=\pi(s). $
\end{center}

In our work we will  use the characterization of a partial action of an inverse semigroup given in \cite[Proposition~3.5]{Buss}: \\




\begin{definition}\label{acaosemigroup2} 
Let $X$ be a set and $\pi : S \rightarrow I(X)$ be a map. For each $s \in S,$ let $X_s$ be the range of $\pi_s.$ Then $\pi$ is a  partial action of $S$ on $X$ if, and only if, for all $s, t \in S,$  the following holds:

\begin{enumerate}[{\rm (i)}]
\item\label{A1} $\pi_s^{-1}=\pi_{s^*}$  (in particular this implies that the domain of $\pi_s$ coincides with the range of $\pi_{s^*},$  namely $X_{s^*});$
\item\label{A2} $\pi_s(X_{s^*} \cap X_t ) \subseteq X_s \cap X_{st};$ 
\item\label{A3} if $s \leq t,$ then $X_s \subseteq X_t,$
\item\label{A4} for all $x \in X_{t^*}\cap X_{t^*s^*}$, it holds that $\pi_s(\pi_t(x))=\pi_{st}(x).$
\end{enumerate}
\end{definition}

By Proposition~3.4 in \cite{Buss} we can replace inclusion in (\ref{A2}) by the equality 
\begin{equation}\label{eqaction}
\pi_s(X_{s^*} \cap X_t )=X_s \cap X_{st},
\end{equation}
and discard (\ref{A3}).

An inverse semigroup $S$ may partially act on a topological space, an algebra, a C*-algebra, among other objects. For each situation we add some conditions in the definition of a partial action to be in accordance with the characteristics of each category. For example, for a partial action $\pi$ of an inverse semigroup $S$ into a $R$-algebra $A,$ we assume that every $S,$ $(s \in S),$ is a bilateral ideal of $A$, and that each application $\pi_s: X_{s^*} \rightarrow X_s$ is an isomorphism of algebras. In addition, if the inverse semigroup $S$ has a unit 1, we assume that $X_1 = A$ and $\pi_1$ is the identity application of $A.$ In the case that $\pi $ is a partial action of $S$ in a topological space $X,$ we also require that, for each $s \in S, $ the subsets $X_s$ are open, and that each application $\pi_s: X_{s^*} \rightarrow X_s $ is a homeomorphism of topological spaces. In the same way as in the above case, if $S$ has a unit, we assume that $X_1 = X$ and $\pi_1$ is the identity map in $X.$

It well known that given a partial action $\theta$ of a group $G$, in a locally compact Hausdorff topological space $X$, there is an associated partial action $\alpha$ of $G$ in $C_0(X)$ (see \cite[Corollary~11.6]{ExelPartial}). Similarly we have the following result:

Given an inverse semigroup $S,$ a Hausdorff, locally compact, totally disconnected topological space $X,$ and a partial action $\theta: S \rightarrow I(X),$  such that each $X_s$ is clopen, there exist a partial action  $\alpha: S \rightarrow I(\Lc),$ where $\Lc$ is the $R$-algebra of the locally constant functions of $X$ in $R$ with compact support. Such action is constructed in the following way:

For each $s \in S,$ define 
  $$D_s:= \{ f \in \Lc \, : \, f(x)=0, \, \forall \, x  \in X \setminus X_s  \}.$$
It is easy to see that $D_s$ is a bilateral ideal of $\Lc.$ Note that  $D_s$ can be identified with $\Lcs.$ Now, let
$\alpha_s: D_{s^*} \rightarrow D_s$ be given by 
  $$ \alpha_s(f)(x)= \left\lbrace \begin{array}{ccr}
    f(\theta_{s^*}(x)), & & \mbox{se} \,\, x \in X_s \\
    0,                       & & \mbox{se} \,\, x \notin X_s  
    \end{array}\right..
  $$

\subsection{Partial skew inverse semigroup rings}\label{sectionskew}

In this section we recall the concepts related with partial skew inverse semigroup rings, as defined in \cite{Vieira}.

\begin{definition} Let $K$ be a commutative ring with unit, and let $\pi$ be a partial action of the inverse semigroup $S$ in the $R$-algebra $A.$ For each $s \in S,$ let $D_s$ be the range of $\pi_s.$ Consider $\mathcal{L}$ the $K$-module of all finite sums of the form
 $$ \sum_{s \in S}^{finite} a_s \delta_s,  \,\,\, a_s \in D_s ,$$
with product being the linear extension of
\begin{equation}\label{multiplicacaoL}
(a_s \delta_s)(a_t \delta_t) = \alpha_s(\alpha_{s^*} (a_s)a_t)\delta_{st}.
\end{equation}
\end{definition}

Notice that  $\pi_s(\pi_{s^*} (a_s)a_t)\delta_{st}$ is an element in $D_{st}$ since
  $$\pi_s(\pi_{s^*} (a_s)a_t) \in \pi_s(D_{s^*}\cap D_t)=D_s \cap D_{st}.$$
Hence the product is well defined and $\mathcal{L}$ is a $K$-algebra (not necessarily associative).

\begin{remark}In the case of a group action, the $K$-algebra $\mathcal{L}$ is exactly the partial skew group ring. In \cite{Dokuchaev1}, Dokuchaev and Exel prove under which conditions the partial skew group ring $A \rtimes G$ is associative. Analogously results are proved in \cite{Vieira} for $\mathcal{L}$ (and consequently for $\Skew$ defined below).
\end{remark}


\begin{definition} Let $\pi = ( \{\pi_s\}_{s \in S}, \{D_s\}_{s \in S})$ be a partial action of an inverse semigroup $S$ on a $K$-algebra $A$ and let $\mathcal{I} =  \left\langle  a\delta_s-a\delta_t \ : \ a \in D_s,  s \leq t \right\rangle,$ that is, $I$ is the ideal generated by $a\delta_s-a\delta_t.$  We define the \emph{partial skew inverse semigroup ring} of $\pi$ as 
 $$\Skew = \dfrac{\mathcal{L}}{\mathcal{I}}.$$
We denote the elements of $\Skew$ by $\overline{a_s\delta_s},$ where  $a_s\delta_s \in \mathcal{L}.$
\end{definition}

Note that by Proposition~\ref{acaosemigroup2}~(\ref{A3}), $s \leq t$ implies $D_s \subseteq D_t.$ Thus, if $s \leq t$ and $a \in D_s$ then $\overline{a\delta_s}= \overline{a\delta_t}.$







\begin{remark} Notice that if the inverse semigroup $S$ is a group then the order relation coincides with equality, and hence the ideal $\mathcal{I} = \{0\}$. Therefore partial skew inverse semigroup generalize partial skew group rings. However, the more general structure of partial skew inverse semigroup rings do not carry the graded structure present in the group case. 
Instead, we only have that every partial skew inverse semigroup ring admits a pre-grading (defined below) over the semigroup. 
\end{remark}

\begin{definition} 
Let $B$ be any $K$-algebra and let $S$ be an inverse semigroup. A pre-grading of $B$ over $S$ is a family of linear subspaces $\{B_s\}_{s \in S}$ of $B,$ such that for every $s, t \in S$ one has that
\begin{enumerate}
\item $B_sB_t \subseteq B_{st},$
\item if $s \leq  t$  then $B_s \subseteq B_t,$
\item $B$ is the linear span of the union of all $B_s.$
\end{enumerate}
\end{definition}

\section{The Steinberg algebra of a transformation groupoid}

In this section we prove that partial skew group rings of the form $\Lg$ can be realized as Steinberg algebras. This result is attributed to folklore and we provide a proof of it here. Throughout we assume that $R$ is a commutative ring with unit, and $X$ is a Hausdorff, locally compact, totally disconnected topological space. 

Let $\Lc$ be the set of all locally constant, compactly supported, $R$-valued functions on $X$. Notice that $\Lc$ is a commutative $R$-algebra, with point-wise addition and multiplication. Furthermore, $\Lc$ has an unit if, and only if, $X$ is compact. For $f \in \Lc,$ we define the support of $f$ by $\supp(f)= \{x \in X \ :  \ f(x)\neq 0\}$ (notice that this is a clopen set).

Notice that the topology of $X$ admits a basis formed  by  subsets which are simultaneously compact and open (compact-open). Given any compact-open set $K \subseteq X$,  its characteristic function, which we denote by $1_K,$ is locally constant and compactly supported. Moreover, one may prove that every locally constant, compactly supported function $f : X \rightarrow R$ is a linear combination of the form 
$$f= \sum_{i=1}^n r_i1_{K_i},$$
where the $K_i$ are pairwise disjoint compact-open subsets, and each $r_i$ lies in $R.$

Now, let $\theta=(\{X_g\}_{g \in G}, \{\theta_g\}_{g \in G})$ be a partial action of a discrete group $G$ on $X,$ such that $X_g$ is clopen (closed-open) for every $g$ in $G.$ Such action induces an action in the algebra level, as done in \cite{Beuter} and \cite{Dokuchaev}: For each $g$ in $G,$ consider the ideal $D_g:= \{ f \in \Lc \, :  \, f \ \mbox{ vanishes on } X\setminus X_g \}$ in $\Lc,$ and define $\alpha_g: D_{g^{-1}} \rightarrow D_g$ by setting $\alpha_g(f)=f\circ \theta_{g^{-1}},$ for all $ f \in D_{g^{-1}}.$ Then the collection 
\begin{equation}
\alpha := ( \{D_g\}_{g \in G}, \{\alpha_g\}_{g \in G} )
\end{equation}
is an algebraic partial action of $G$ on $\Lc.$ 

Associated to the above partial action we consider the partial skew group ring 
\begin{equation}
\Lg.
\end{equation}
A general element $b \in \Lg$ is denoted by 
 $$b=\sum_{g \in G}f_g\delta_g,$$
where each $f_g$ lies in $D_g,$  and $f_g \equiv 0,$ for all but finitely many group elements $g.$ 


We can also associate to an action $\theta$ as above an \'{e}tale groupoid, denoted by $\Gx,$ and known as the \emph{transformation grupoid}  (in what follows we use the same simplified notation of the transformation groupoid as in \cite{Abadie}): Let
 $$\Gx:= \{(t,x) \, : \, t \in G \,\, \mbox{and} \,\, x \in X_t \}.$$
The inverse of $(t, x) \in \Gx$ is 
$$(t, x)^{-1}=(t^{-1}, \theta_{t^{-1}}(x)).$$
If $(s,y), (t,x) \in \Gx$ then $(s,y), (t,x) $ is a composable pair if, and only if, $\theta_{s^{-1}}(y)=x.$ In this case we have
$$(s,y)(t,x)=(st, y).$$

We equip $\Gx$ with the topology inherited from the product topology on $G \times X.$  By the continuity of $\theta$ the
operations of inversion and product on $\Gx$ are continuous. Since $G$ is discrete and $X$ is Hausdorff we have that $\Gx$ is Hausdorff. 

Notice that the range and source maps $r: \Gx \rightarrow  (\Gx)^{(0)}$ and $s : \Gx \rightarrow (\Gx)^{(0)}$ are given by $r(t,x) =(1,x)$, and $s(t,x)=(1, \theta_{t^{-1}}(x)).$ Moreover, the range is a local homeomorphism of $\{t\} \times X_t$ onto $\{1\} \times X_t$  and the source is local homeomorphism of $\{t\} \times X_t$ onto $\{1\} \times X_{t^{-1}}.$ We identify $X$ with $(\Gx)^{(0)}$ via the homeomorphism $i: X \rightarrow \Gx$ given by $x \mapsto (1, x)$. Thus, $(\Gx)^{(0)}$ is Hausdorff, locally compact and totally disconnected, and $\Gx$ is \'{e}tale. Furthermore, by \cite[Proposition~4.1]{Exel2}, $\Gx$ is ample and Hausdorff and we may consider its Steinberg algebra, $A_R(\Gx)$.

\begin{remark} Observe that the convolution product in $A_R(\Gx)$ is defined by 
\begin{align*}
f_1*f_2(t,x) & = \sum_{(s,x)\in \Gx}f_1(s,x)f_2((s,x)^{-1}(t,x)) \\
             & = \sum_{(s,x)\in \Gx}f_1(s,x)f_2(s^{-1}t,\theta_{s^{-1}}(x)),
\end{align*}
for all $f_1, f_2 \in A_R(\Gx)$ and $(t,x) \in \Gx.$
\end{remark}

We can now prove the following.

\begin{theorem}\label{theorisomortransfgrou}
Let $\theta=(\{X_g\}_{g \in G}, \{\theta_g\}_{g \in G})$ be a partial action of a discrete group $G$ over a locally compact, Hausdorff, totally disconnected topological space $X$, such that each $X_g$ is clopen. Let $(\{D_g\}_{g \in G}, \{\alpha_g\}_{g\in G})$ be the corresponding partial action (as defined above) and $\Gx$ be the transformation groupoid associate with $\theta.$ Then $\Lg$ and $A_R(\Gx)$ are isomorphic as $R$-algebras. 
\end{theorem}

\begin{proof}
To define a homomorphism $\rho$ of $\Lg$ to $A_R(\Gx),$  we begin defining it in elements of the form $f_g\delta_g$, and then we extend it linearly to $\Lg.$ More precisely, for $f_g\delta_g \in \Lg$ and $(t,x) \in \Gx,$ let
 $$\tilde{\rho}(f_g\delta_g)(t,x):= \left\lbrace\begin{array}{ccl}
 f_g(x), & & \mbox{if} \,\, t=g \\
 0 , & & \mbox{otherwise},
 \end{array}  \right.$$
and denote the linear extension of  $\tilde{\rho}$ to $\Lg$ by $\rho.$

We claim that $\rho$ is well defined. For this it is enough to prove that, for each $g\in G$, $\tilde{\rho}(f_g\delta_g)$ is well defined, that is, it is a locally constant function with compact support. Let $(t,x) \in \Gx.$ Suppose that $t=g.$ Since $f_g$ is locally constant, there is a neighborhood $U \subseteq X$ of $x$ such that $ f|_U$ is constant. Note that
$$V:= (\{g\} \times U) \cap (\Gx)$$
is a neighborhood of $(t,x)$ and $\tilde\rho(f_g\delta_g)|_V=f_g|_U,$ which is constant. 
Now, suppose that $t \neq g.$ Then 
 $$W:=((G \setminus \{g\}) \times X ) \cap (\Gx)$$
is a neighborhood of $(t,x)$ in $\Gx$ and $\tilde\rho(f_g\delta_g)|_W \equiv 0.$ So $\tilde{\rho}(f_g\delta_g)$ is locally constant.  We easily see that 
 $$\mbox{supp}(\tilde\rho(f_g\delta_g))=\{g\} \times \mbox{supp}(f_g),$$
which is compact. We concluded that  $\tilde\rho(f_g\delta_g) \in A_R(\Gx).$
 
Next we check that $\rho$ is multiplicative. By linearity, it is enough to check this for elements of the form $f_g\delta_g \in \Lg.$ So, let $f_g\delta_g$ and $f_h\delta_h \in \Lg$ and $(t,x) \in \Gx.$ Then,

\begin{align*}
\rho((f_g\delta_g)(f_h\delta_h))(t,x) & = \rho(\alpha_g(\alpha_{g^{-1}}(f_g)f_h)\delta_{gh})(t,x) \\
                                      & = \left\lbrace\begin{array}{cl}
                                               \alpha_g(\alpha_{g^{-1}}(f_g)f_h)(x),  & \mbox{if} \,\, t=gh \\
                                               0, & \,\,  \mbox{otherwise}
                                           \end{array}  \right. \\
                                      & = \left\lbrace\begin{array}{cl}
                                               \alpha_{g^{-1}}(f_g)f_h(\theta_{g^{-1}}(x)),  & \mbox{if} \,\, t=gh \\
                                               0, & \,\,  \mbox{otherwise}
                                           \end{array}  \right. \\
                                      & = \left\lbrace\begin{array}{cl}
                                              \alpha_{g^{-1}}(f_g)(\theta_{g^{-1}}(x))f_h(\theta_{g^{-1}}(x)), & \mbox{if}\,\, t=gh \\
                                              0, & \,\,  \mbox{otherwise}
                                          \end{array}  \right. \\
                                      & = \left\lbrace\begin{array}{cl}
                                               f_g(x)f_h(\theta_{g^{-1}}(x)), & \mbox{if} \,\, h=g^{-1}t \\
                                               0, & \,\,  \mbox{otherwise}                                               
                                           \end{array}  \right. \\ 
                                      & = f_g(x)\rho(f_h \delta_h)(g^{-1}t, \theta_{g^{-1}}(x)) \\
                                      & = \sum_{(s,x) \in \G}\rho(f_g\delta_g)(s,x)\rho(f_h\delta_h)(s^{-1}t,\theta_{s^{-1}}(x))\\
                                      & = \rho(f_g\delta_g)* \rho(f_h\delta_h)(t,x).                                                                                         
\end{align*}
Therefore $\rho$ is a homomorphism.

To finish we prove that $\rho$ is a bijection, by showing that it has an inverse $\rho^{-1}: A_R(\Gx) \rightarrow \Lg$ given by $\rho^{-1}(f)=\sum f_g\delta_g,$  where
 $$f_g(x):= \left\lbrace \begin{array}{lcl}
 f(g, x), & & \mbox{if} \,\,\,\, x \in X_{g} \\
 0, & & \mbox{otherwise.}
 \end{array} \right.$$

We claim that $\rho$ is well defined. For this we need first to prove that $f_g \in D_g.$

Since the topology of $\Gx$ is the relative product topology from  product topology of $G \times X$ and  $\supp(f)$ is compact, we have  $\supp(f)=(\{g_1, \cdots , g_n \} \times K) \cap (\Gx),$ where $\{g_1, \cdots , g_n\} \subseteq G$ and $K \subset X$ is compact. Thus $\supp(f_g)= K \cap X_g$ is compact.

If $x \notin X_g$ then $f$ is identically null in the open subset $X \setminus X_g.$ Let $x \in X_g.$ Then there is an open neighborhood $W$ of $(g,x)$ such that $f|_W$ is constant. Notice that $W$ is of the form  $(H \times U)\cap (\Gx), $ where $H$ is open subset of $G$ and $U$ is open subset of $X.$ We have that $V := U \cap X_g$ is an open neighborhood of $x$ and $f_g|_V$ is constant. 

Obviously $f_g(x)=0$ for every $x \notin X_g.$ Therefore $f_g \in D_g.$

Notice that the set $\{g \in G \ : \  f(g,x) \neq 0, \ x \in X_g \}$ is finite because $G$ is discrete and $\supp(f)$ is compact. Then the set $\{g \in G \ : \ f_g \neq 0\}$ also is finite, that is, the sum  $\sum f_g\delta_g = \rho^{-1}(f)$ is finite. 

It is straightforward to check that $\rho^{-1}$ is the inverse of $\rho.$
\end{proof}

\begin{remark} Under the assumptions of the theorem above we remark that partial actions such that each $X_g$ is clopen are exactly the one's for which the envelope space is Hausdorff (see \cite{EGG}).
\end{remark}

\begin{corollary}\label{coroldiagonal}
The isomorphism above establishes an isomorphism between the algebras $\Lc $ and $D_R(\Gx).$ 
\end{corollary}

\begin{proof}
Consider 
$$i(A_R(\Gx)^{(0)}): = \{f \in A_R(\Gx) \ : \ \supp(f) \subseteq (\Gx)^{(0)} \}$$ 
the embedding of $D_R(\Gx)= A_r((\Gx)^{(0)})$ in $A_R(\Gx)$ as given in (\ref{embeddingdiagonal}).

It is clear that the restriction of $\rho$ to $\Lc \delta_1$ is an injective homomorphism in $i(A_R(\Gx)^{(0)}).$
For surjectivity, notice that every compact bisection of $\Gx$ contained in $(\Gx)^{(0)}$ is of form $\{1\} \times K,$ where $K$ is a compact subset of  $X.$ Thus, given $f \in i(A_R((\Gx)^{(0)}))$ there exists compacts subset $K_1, \cdots, K_n$ of $X$ and elements $r_1, \cdots, r_n \in R $ such that  $f=\sum_{i=1}^n r_i1_{\{1\}\times K_i}.$ Now, consider  $F:= \left( \sum_{i=1}^n r_i1_{K_i} \right) \delta_1 \in \Lc.$  Then $\rho(F)=f$ and hence $\rho$ is surjective. Therefore we have an isomorphism between $\Lc\delta_1$ and $i(A_R((\Gx)^{(0)})$, and by identifying these sets with $\Lc$ and $D_R(\Gx),$ respectively, we obtain the desired isomorphism.
\end{proof}

\section{Application to diagonal preserving isomorphisms}

In this section we study diagonal preserving isomorphisms of partial skew group rings of the form $\Lg$. For this purpose we apply the isomorphism of the previous section and we make use of results proved in \cite{James} and \cite{Xin}.
We continue to use the same assumptions on $X, G$ and $\theta$ of the previous section. 

By Theorem~\ref{theorisomortransfgrou} and Corollary~\ref{coroldiagonal}, we obtain an isomorphism between the algebras $\Lg$ and $A_R(\Gx)$ that \an{preserves diagonal}. In \cite[Theorem 3.1.]{James}, Carlsen and Rout characterize diagonal preserving (graded) isomorphism between two (graded) Steinberg algebras. For our particular case we will use \cite[Corollary~3.2]{James} as follows:

\begin{corollary}\label{steinbdiagonal}
Let $R$ be an integral domain (commutative ring, with no zero-divisors and with a unit). For $i=1,2$, let $\G_i$ be an ample Hausdorff groupoid such that there is a dense subset $X_i \subseteq \G_i^0$, such that the group-ring $R((\G_i)_x^x)$ has no zero-divisors and only trivial units for all $x \in X_i$. Then $\G_1$ and $\G_2$ are isomorphic if and only if there is a diagonal-preserving isomorphism between $A_R(\G_1)$ and $A_R(\G_2)$.
\end{corollary}

Recall that if $R$ is an integral domain and  $G$ is a group, the group-ring $RG$ of $G$ is defined by 
$$RG := \left\lbrace \sum_{i=1}^n r_ig_i \ : n \in \N, r_i \in R \ \mbox{and} \ g_i \in G \right\rbrace.$$
An element $x \in RG$ is a \emph{unit} if there exist $y, z \in RG$ such that $xy = 1 = zx.$ A unit is \emph{trivial} if it has the form $rg$ for some $r \in R$ and $g \in G.$

By the previous corollary, we need to find sufficient (and necessary, if possible) conditions so that the group-ring $R((\Gx)_{(1,x)}^{(1,x)})$ has no zero-divisors and only trivial units. We have that
\begin{align*}
(\Gx)_{(1,x)}^{(1,x)} & = \{(t,y) \in \Gx \ : \ r(t,y)=(1,x)=s(t,y)\} \\
                      & = \{(t,y) \in \Gx \ : \ (1,y)=(1,x)=(1,\theta_{t^{-1}}(y))\} \\
                      & = \{(t,x) \in \Gx \ : \ x=\theta_{t^{-1}}(x) \}. 
\end{align*}

Notice that the group $(\Gx)_{(1,x)}^{(1,x)}$ is isomorphic to a subgroup of $G.$ To study this subgroup we present some results of \cite{Higman} below.

An infinite group $G$ is said to be \emph{indexed} if we are given a homomorphism $\gamma$ of $G$ in the additive group of integers, such that $\gamma(G)$ does not consist of zero alone. In general, a group can be indexed in more than one way. An infinite group $G$ is said to be \emph{indicable throughout} if every subgroup of $G$, not consisting of the unit alone, can be indexed. Notice that any non-trivial subgroup of an indicable throughout group is also indicable throughout. Since any free group is indexed, and any subgroup of a free group is either a free group or the unit alone, we have that any free group is indicable throughout. Similarly any free Abelian group is indicable throughout.

\begin{theorem}\label{thoRGzerodivisores}
\cite[Theorem~12]{Higman} If $G$ is indicable throughout and $R$ has no zero-divisors then $RG$ has no zero-divisors.
\end{theorem}

\begin{theorem}\label{thoRGunitstrivial}
\cite[Theorem~13]{Higman} If $G$ is indicable throughout and $R$ has no zero-divisors then all the units of $RG$ are trivial.
\end{theorem}

With these two theorems we obtain the following lemma:

\begin{lemma}\label{lemmaindicable}
Let $G$ be an indicable throughout group and let $\theta=(\{X_g\}_{g \in G}, \{\theta_g\}_{g \in G})$ be a partial action of $G$ on $X.$ Then the group-ring $R((\Gx)_{(1,x)}^{(1,x)})$ has no zero-divisors and only trivial units, for all $(1,x) \in (\Gx)^{(0)}.$
\end{lemma}

\begin{proof}
We have that $(\Gx)_{(1,x)}^{(1,x)}$ is isomorphic to a subgroup of $G.$ Then this subgrupo is indicable throughout or the unit alone. In the first case, Theorems~\ref{thoRGzerodivisores} and \ref{thoRGunitstrivial} ensure that $R((\Gx)_{(1,x)}^{(1,x)})$ has no zero-divisors and only trivial units. In the second case, $R((\Gx)_{(1,x)}^{(1,x)})=R(\{(1,x)\}) \cong R,$ that also has no zero-divisors and only trivial units.
\end{proof}

\begin{theorem}\label{theodiagindicable}
Let $R$ be an integral domain, let $G, H$ be indicable throughout, discrete groups, let $X, Y$ be totally disconnect, locally compact, Hausdorff spaces, and let $\theta=(\{X_g\}_{g \in G}, \{\theta_g\}_{g \in G})$ and $\gamma=(\{Y_h\}_{h \in H}, \{\theta_h\}_{h \in H})$ be partial actions. Then the following are equivalent:
\begin{enumerate}[{\rm (i)}]
\item the transformation groupoids $\Gx$ and $H \ltimes_{\gamma} Y$ are isomorphic as topological groupoids,
\item there exists a diagonal-preserving isomorphism $\Gamma: \ A_R(\Gx) \longrightarrow  A_R(H \ltimes_\gamma Y),$ 
\item there exists an isomorphism $\Phi: \ \mathcal{L}_c(X)\rtimes G \longrightarrow \mathcal{L}_c(Y)\rtimes H$ with $\Phi(\Lc) = \mathcal{L}_c(Y)$.
\end{enumerate}
\end{theorem}

\begin{proof}
\an{$(i) \Leftrightarrow (ii)$} By Lemma~\ref{lemmaindicable}, the hypotheses of Corollary~\ref{steinbdiagonal} are satisfied and thus we get the two implications.

\an{$(ii) \Rightarrow (iii)$} Let $\rho_G: \mathcal{L}_c(X)\rtimes G \longrightarrow A_R(\Gx)$ and $\rho_H : \mathcal{L}_c(Y)\rtimes H \longrightarrow  A_R(H \ltimes_\gamma Y)$ be the isomorphisms given by Theorem~\ref{theorisomortransfgrou}. Define $\Phi: \ \mathcal{L}_c(X)\rtimes G \longrightarrow \mathcal{L}_c(Y)\rtimes H$ by $\Phi:= \rho_H \circ \Gamma \circ \rho_G^{-1}.$ Clearly $\Phi$ is an isomorphisms and
\begin{align*}
\Phi(\Lc) & = \rho_H \circ \Gamma \circ \rho_G^{-1}(\Lc) \\
          & \stackrel{\ref{coroldiagonal}}{=}\rho_H \circ \Gamma (D_R(\Gx)) \\
          & \stackrel{(iii)}{=} \rho_H (D_R( H \ltimes_\gamma Y))\\
          & \stackrel{\ref{coroldiagonal}}{=} \mathcal{L}_c(Y).
\end{align*}
\an{$(iii) \Rightarrow (ii)$} Similar to the previous one, just take $\Gamma:=\rho_H^{-1} \circ \Phi \circ \rho_G.$ 
\end{proof}

In \cite[Theorem~2.7]{Xin} X. Li characterizes diagonal preserving isomorphisms of partial C*-crossed products, over commutative algebras, in terms of continuous orbit equivalence of the associated partial actions. We are now able to add two additional equivalent conditions to continuous orbit equivalence of partial actions, in terms of Steinberg algebras and partial skew group rings. Before we do this we recall the notion of continuous orbit equivalence below.

\begin{definition}
Partial actions $\theta=(\{X_g\}_{g \in G}, \{\theta_g\}_{g \in G})$ and $\gamma=(\{Y_h\}_{h \in H}, \{\gamma\}_{h \in H})$ are called continuously orbit equivalent if there exists a homeomorphism $\varphi: \: X \stackrel{\cong}{\longrightarrow} Y$, and continuous maps $a: \: \bigcup_{g \in G} \{g\} \times X_{g^{-1}} \longrightarrow H$, $b: \: \bigcup_{h \in H} \{h\} \times Y_{h^{-1}} \longrightarrow G$  such that
\begin{enumerate}[{\rm (i)}]
\item $ \varphi(\theta_g(x)) = \gamma_{a(g,x)}(\varphi(x)), $
\item $ \varphi^{-1}(\gamma_h(y)) = \theta_{b(h,y)}(\varphi^{-1}(y)).$
\end{enumerate}
Implicitly, we require that $a(g,x) \in H_{\varphi(x)}:=\{ h \in H \ : \ \varphi(x) \in X_{h^{-1}}\}$ and $b(h,y) \in G_{\varphi^{-1}(y)}:= \{ g \in G \ : \ \varphi^{-1}(y) \in  X_{g^{-1}}\}$.
\end{definition}

Recall that a partial action $\theta=(\{X_g\}_{g \in G}, \{\theta_g\}_{g \in G})$ is called \emph{topologically free} if, for every $1 \neq g \in G,$ $\{x \in X_{g^{-1}} \ :  \ \theta_g(x) \neq x \}$ is dense in $X_{g^{-1}}.$

\begin{lemma}\label{lemmafreeprincipal}
A partial action $\theta$ of $G$ on $X$ is topologically free if, and only if, the transformation groupoid $\Gx$ is topologically principal. 
\end{lemma}

\begin{proof}
See \cite[Lemma~2.4]{Xin}.
\end{proof}




Before we state our next Theorem we need the following lemma.

\begin{lemma}\label{lemmagroupring}
Let $R$ be an integral domain and $\theta$ be a partial action of $G$ on $X.$ If $\theta$ is topologically free then there exist a dense subset $Z \subseteq  (\Gx)^{(0)}$ such that the group-ring $R((\Gx)_{z}^{z}))$  has no zero-divisors and only trivial units for all $z \in Z.$
\end{lemma}

\begin{proof}
By Lemma (\ref{lemmafreeprincipal}) the groupoid $\Gx$ is topologically principal, that is, there exist a dense subset $Z$ of $(\Gx)^{(0)}$ such that $(\Gx)_{z}^{z}= \{z\},$ for every $z \in Z.$ Thus for $z \in Z$ we have that  $R((\Gx)_{z}^{z})=R(\{z\}) \cong R,$ which has no zero-divisors and only trivial units.
\end{proof}

\begin{theorem}
Let $R$ be an integral domain, let $G, H$ be discrete groups, let $X, Y$ be totally disconnect, locally compact, Hausdorff spaces, and let $\theta=(\{X_g\}_{g \in G}, \{\theta_g\}_{g \in G})$ and $\gamma=(\{Y_h\}_{h \in H}, \{\theta_h\}_{h \in H})$ be topologically free partial actions. Then the following are equivalent:
\begin{enumerate}[{\rm (i)}]
\item $\theta$ and $\gamma$ are continuously orbit equivalent,
\item the transformation groupoids $\Gx$ and $H \ltimes_{\gamma} Y$ are isomorphic as topological groupoids,
\item there exists a diagonal-preserving isomorphism $\Gamma: \ A_R(\Gx) \longrightarrow  A_R(H \ltimes_\gamma Y),$ 
\item there exists an isomorphism $\Phi: \ \mathcal{L}_c(X)\rtimes G \longrightarrow \mathcal{L}_c(Y)\rtimes H$ with $\Phi(\Lc) = \mathcal{L}_c(Y)$,
\item there exists an isomorphism $\Phi: \: C_0(X) \rtimes_r G \longrightarrow C_0(Y) \rtimes_r H$ with $\Phi(C_0(X)) = C_0(Y)$.
\end{enumerate}
Moreover, \an{(ii) $\Rightarrow$ (i)} holds in general (i.e., without the assumption of topological freeness).
\end{theorem}

\begin{proof} 
\an{$(i) \Leftrightarrow (ii)$} See \cite[Theorem 2.7]{Xin}.

\an{$(ii) \Leftrightarrow (iii)$} By Lemma (\ref{lemmagroupring}), the hypotheses of \cite[Corollary~3.2.]{James} are satisfied and hence we get the desired implications.

\an{$(iii) \Leftrightarrow (iv)$} Analogous to the proof of \an{$(ii) \Leftrightarrow (iii)$} in Theorem~\ref{theodiagindicable}.

\an{$(iv) \Leftrightarrow (i)$}  See \cite[Theorem 2.7]{Xin}.
\end{proof}

\section{Steinberg algebras realized as partial skew inverse semigroup rings }

In this section we will show that every Steinberg algebra can be realized as a partial skew inverse semigroup ring. This is an “algebraisation” of [Theorem~3.3.1]\cite{Paterson} and [Theorem 8.1]\cite{Quigg}.

\subsection*{The partial action of the inverse semigroup of bisections}\label{PAISGBi}

Next we recall the example of an inverse semigroup action which is intrinsic to every Hausdorff ample groupoid, as given in \cite{Exel}. 

From now on we fix a Hausdorff ample groupoid $\G$. Denote by $\Ga$ the set of all compact bisection in $\G.$ We have already seen, Proposition~\ref{bisectsetsemigroup}, that $\Ga$ is an inverse semigroup under the operations defined in (\ref{bisectionproduct}) and (\ref{bisectioninverse}). The idempotent semilattice of $\Ga$ consists precisely of the compact-open subsets of $\Go.$

Note that the inverse semigroup partial order (as in (\ref{semigrouporder})) in $\Ga$ is defined by
$$A \leq B \,\,\,\, \Leftrightarrow A \subseteq B, $$
where $A, B \in \Ga.$

Recall that the unit space $\Go$ is a Hausdorff, locally compact, totally disconnected topological space.  We want to define a partial action $\theta$ of $\Ga$ on $\Go.$ Given a compact bisection $B$ we have already mentioned that $s(B)$ and $r(B)$ are compact-open subsets of $\Go,$ and moreover that the maps
  $$ s_B : B \to s(B) \,\,\,\,\,   \mbox{e} \,\,\,\,\,   r_B : B \to r(B), $$
obtained by restricting $s$ and $r,$ respectively, are homeomorphisms. Given $u \in s(B)$ we let
\begin{equation}\label{acaosemigrupobissecao}
      \theta_B(u) := r_B(s_B^{-1}(u)). 
\end{equation}

\begin{center}
\fbox{
\begin{tikzpicture}
\node (y) at (0,0) {$s_B^{-1}(x)$};\filldraw(y.south) circle (1pt);
\node[fit=(y), circle, draw, minimum width=1.7cm, thick, label=above:\(B\)]{};

\node (z) at (3,-2) {{\scriptsize $r(s_B^{-1}(x))$}};\filldraw(z.west) circle (1pt);
\node[fit=(z), circle, draw, minimum width=1.5cm, thick, label=above:\(r(B)\)]{};

\node (x) at (-3,-2) {$x$};\filldraw(x.east) circle (1pt);
\node[fit=(x), circle, draw, minimum width=1.7cm, thick, label=above:\(s(B)\)]{};

\node (o) at (0,-3) {$\theta_B$};

\draw[->, shorten >=.1cm, >=stealth'](x.east) to [out=40, in=170](y.south);
\draw[->, shorten >=.1cm, >=stealth'](y.south) to [out=10, in=120] (z.west);
\draw[->, shorten >=.1cm, >=stealth'](x.east) to [out=-20, in=210] (z.west);
\end{tikzpicture}}
\end{center}

Clearly $\theta_B$ is a homeomorphism from $s(B)$ to $r(B).$ Here we have $B^*=B^{-1},$ $X_{B^*}=s(B)$ and $X_B=r(B).$

\begin{prop}\label{PABisections} 
The correspondence $B \mapsto \theta_B$, where $\theta_B$ is defined in (\ref{acaosemigrupobissecao}), gives a partial action of $\Ga$ on the unit space $\Go.$
\end{prop}

\begin{proof}
This is proved in \cite[Proposition~5.3]{Exel}. For completeness we present a proof adapted to our setting below. 

To show that $\theta$ is a partial action we will use the characterization of Definition~\ref{acaosemigroup2}. 

If $v \in s(B^*)=s(B^{-1})=r(B),$ then
 $$\theta_{B^*}(v)=\theta_{B^{-1}}(v)=r_{B^{-1}}(s_{B^{-1}}^{-1}(v))= s_B(r_B^{-1}(v))=(\theta_{B}(v))^{-1},$$
and hence $\theta_{B^*}=\theta_B^{-1},$ proving item (\ref{A1}). 

For item (\ref{A2}), we need to prove that 
$$\theta_B(s(B)\cap r(C)) \subseteq r(B) \cap r(BC).$$
Let $u \in s(B) \cap r(C).$ Clearly $\theta_B(u) \in r(B).$ Since $u \in s(B) \cap r(C),$  there exists  $b \in B$ and $c \in C$ such that  $s(b)= u = r(c).$ Notice that  $bc \in BC$ and  $b$ is the only element of the bisection $B$ such that $s_B(b)=u.$

Thus $s_B^{-1}(u)=b$ and
 $$ \theta_B(u)=r_B(s_B^{-1}(u))=r_B(b)=bb^{-1}=r(bc) \in r(BC).$$

Item (\ref{A3}) follows because  $B \leq C$ if, and only if,  $B \subseteq C.$ 

For item (\ref{A4}) suppose that $u \in s(B)\cap s(CB)$, and that $\theta_C \theta_B(u)=v.$ Then there exists $b \in B, w \in r(B)\cap s(C)$ and $c \in C$ such that
  $$s(b)=u, \,\,\,\,\,\,\,\,\,\,\,\, r(b)=w, \,\,\,\,\,\,\,\,\,\,\,\, s(c)=w \,\,\,\,\,\,\, \mbox{and} \,\,\,\,\,\,\, r(c)=v.$$ 
  
\begin{center}
\fbox{
\begin{tikzpicture}
\node (w) at (0,0) {$w$};\filldraw(w.east) circle (1pt);
\node[fit=(w), circle, draw, minimum width=1.3cm, thick, label=above left:\(r(B)\)]{};

\node (a) at (0.5,0) {}; 
\node[fit=(a), circle, draw, minimum width=1.3cm, thick, label=above right:\(s(C)\)]{};

\node (c) at (2.5,1.5) {$c$};\filldraw(c.south) circle (1pt);
\node[fit=(c), circle, draw, minimum width=1.3cm, thick, label=above:\(C\)]{};

\node (v) at (4.5,-1) {$v$};\filldraw(v.west) circle (1pt);
\node[fit=(v), circle, draw, minimum width=1.3cm, thick, label=above right:\(r(CB)\)]{};

\node (f) at (4.1,-1) {}; 
\node[fit=(f), circle, draw, minimum width=1.3cm, thick, label=above left:\(r(C)\)]{};

\node (b) at (-2.5,1.5) {$b$};\filldraw(b.south) circle (1pt);
\node[fit=(b), circle, draw, minimum width=1.3cm, thick, label=above:\(B\)]{};

\node (u) at (-4.5,-1) {$u$};\filldraw(u.east) circle (1pt);
\node[fit=(u), circle, draw, minimum width=1.3cm, thick, label=above left:\(s(B)\)]{};

\node (e) at (-4.1,-1) {}; 
\node[fit=(e), circle, draw, minimum width=1.3cm, thick, label=above right:\(s(CB)\)]{};

\node (cb) at (0,-3) {$cb$};\filldraw(cb.north) circle (1pt);
\node[fit=(cb), circle, draw, minimum width=1.3cm, thick, label=below:\(CB\)]{};

\draw[->, shorten >=.1cm, >=stealth'](u.east) to [out=90, in=180] (b.south);
\draw[->, shorten >=.1cm, >=stealth'](b.south) to [out=0, in=120] (w.east);
\draw[->, shorten >=.1cm, >=stealth'](w.east) to [out=90, in=160] (c.south);
\draw[->, shorten >=.1cm, >=stealth'](c.south) to [out=-30, in=90] (v.west);
\draw[->, shorten >=.1cm, >=stealth'](u.east) to [out=-50, in=180] (cb.north);
\draw[->, shorten >=.1cm, >=stealth'](cb.north) to [out=0, in=230] (v.west);
\draw[->, shorten >=.1cm, >=stealth'](u.east) to (v.west);
\end{tikzpicture}}

\end{center}

Since $s(c)=w=r(b),$ we have that $bc \in CB$ and $s(cb)=b^{-1}b=s(b)=u.$ Since $CB$ is a bisection, $cb$ is the only element of  $CB$ such that $s(cb)= u,$ and hence $s_{CB}^{-1}(u)=cb.$ Moreover $r_{CB}(cb)=cc^{-1} = r(c)= v. $ Therefore, 
 $$\theta_{CB}(u)=r_{CB}(s_{CB}^{-1}(u))=r_{CB}(cb)= v = \theta_C \theta_B (u).$$
\end{proof}

As mentioned in Subsection~\ref{PAISG}, from the action $\theta$ of the semigroup $\Ga$ on the locally compact, Hausdorff, totally disconnected space $\Go,$ we get a corresponding partial action $\alpha$ of the semigroup $\Ga$ on the $R$-algebra $\Lco$ of all locally constant, compactly supported, $R$-valued functions on $\Go,$ where $R$ is a commutative ring with unit. More precisely, for every $B \in \Ga,$ we have that $\alpha_B$ is an isomorphism from
 
$$ D_{B^*}=\{f \in \Lco  \ : \ f(x)=0, \, \forall x \in \Go \setminus s(B)\}$$
onto 
$$ D_B=\{f \in \Lco  \ : \ f(x)=0, \, \forall x \in \Go \setminus r(B)\},$$
which is defined by
  $$ \alpha_B(f)(x)= \left\lbrace \begin{array}{ccr}
    f \circ \theta_{B^*}(x), & & \mbox{se} \,\, x \in r(B) \\
    0                        & & \mbox{se} \,\, x \notin r(B)  
    \end{array}\right..
  $$

We can now prove the last theorem of the paper.

\begin{theorem}
Let $\G$ be an ample and Hausdorff groupoid, let $\theta$ be the partial action of the inverse semigroup $\Ga$ over the unit space $\Go$ (as defined in Proposition~\ref{PABisections}), and let $\alpha$ be the corresponding partial action of $\Ga$ on $\Lco$ (as defined above). Then the Steinberg algebra $A_R(\G)$ is isomorphic to the partial skew inverse semigroup ring $\Lgo$ as $R$-algebras.
\end{theorem}

\begin{proof}
We will first show the existence of an epimorphism $\psi: \mathcal{L} \rightarrow A_R(\G),$ that vanishes in the ideal $\mathcal{I}$ (thus, we can extend $\psi$ to an epimorphism $\tilde\psi$ of the quotient $\mathcal{L}/\mathcal{I}= \Lgo$). To conclude, we will show that $\tilde\psi$ admits a right inverse map $\varphi$. 

Define the homomorphism $\psi: \mathcal{L} \rightarrow A_R(\G),$ in elements of the form $f_B\delta_B$, by
$$\psi(f_B\delta_B)(x)= \left\lbrace \begin{array}{ccr}
                                      f_B(r(x)) & & \mbox{if} \,\, x \in B. \\
                                              0 & & \mbox{if} \,\, x \notin B, 
\end{array}\right.$$
and extend it linearly to $\mathcal{L}.$ 

We need to show that $\psi$ is well defined, that is, the function $\psi(f_B \delta_B)$ is locally constant and has compact support.

If $ x \in \G \setminus B$ then $U:= \G \setminus B = B^c $ is an open neighborhood of $x$ and  
  $ \psi(f_B\delta_B)|_U \equiv 0.$
Now, if $x \in B$ then  
  $$\psi(f_B\delta_B)(x)=f_B(r(x)).$$
Since $f_B$ is constant locally, there exist an open neighborhood $V$ of $r(x)$ such that $f_B|_V$ is constant. We can take $V \subseteq r(B)$ because $r(B)$ is open. Thus $U:=r_B^{-1}(V)$ is an open neighborhood of $x$ and   
  $ \psi(f_B\delta_B)|_U =f_B|_V \,\, \mbox{is constant.}$
Therefore, $ \psi(f_B\delta_B)$ is locally constant. Moreover 
  $$ \mbox{supp} ( \psi(f_B\delta_B) ) \subseteq B,$$
and since $\mbox{supp} ( \psi(f_B\delta_B)) $ is closed and $B$ is compact, we have that the support of  $ \psi(f_B\delta_B)$ is compact. Hence $\psi(f_B\delta_B) \in A_R(\G).$ 
Since $A_R(\G)$ is an algebra and any element $F\in \mathcal{L} $ is a finite sum of elements of the form $f_B\delta_B,$ we conclude that $\psi(F)\in A_R(\G).$

Next, we will verify that $\psi$ is multiplicative. By linearity, it is enough to verify that this application is multiplicative in the homogeneous terms. So, let $f_B\delta_B, f_C\delta_C \in \mathcal{L} $ and $x \in \G.$ Then,
\begin{align*}
\psi(f_B\delta_B f_C \delta_C)(x) 
       & = \psi(\alpha_B(\alpha_{B^*}(f_B)f_C)\delta_{BC})(x) \\
       & = \left\lbrace \begin{array}{lcl}
           \alpha_B(\alpha_{B^*}(f_B)f_C)(r(x)), & & \mbox{if} \, x \in BC \\
                                              0, & & \mbox{if} \, x \notin BC
           \end{array}\right.\\
       & \stackrel{r(x) \in r(B)}{=} \left\lbrace \begin{array}{lcl}
           \alpha_{B^*}(f_B)f_C(\theta_{B^*}(r(x)), & & \mbox{if} \, x \in BC \\
                                                 0, & & \mbox{if} \, x \notin BC
           \end{array}\right.\\
       & = \left\lbrace \begin{array}{lcl}
           \alpha_{B^*}(f_B(\theta_{B^*}(r(x))f_C(\theta_{B^*}(r(x)), & & \mbox{if} \, x \in BC \\
                                                                   0, & & \mbox{if} \, x \notin BC
           \end{array}\right.\\
       & \stackrel{\theta_{B^*}(r(x)) \in s(B)}{=} \left\lbrace \begin{array}{lcl}
           f_B(r(x))f_C(\theta_{B^*}(r(x)), & & \mbox{if} \, x \in BC \\
                                         0, & & \mbox{if} \, x \notin BC
           \end{array}\right.\\
       & = \left\lbrace \begin{array}{lcl}
           f_B(r(x))f_C(s_B(r_B^{-1}(r(x)), & & \mbox{if} \, x \in BC \\
                                         0, & & \mbox{if} \, x \notin BC
           \end{array}\right.
\end{align*}
If $x \in BC$ then there exists $b \in B$ and $c \in C$ such that $s(b)=r(c)$ and $x=bc.$ Notice that $r(b)=r(x)$ and $b$ is the only element of $B$ such that $r(b)=r(x) \in r(B).$ Thus
  $$ r_B^{-1}(r(x))=b, \text{ and } s_B(r_B^{-1} (r(x)))= s_B(b)=s(b). $$
Hence 
\begin{align*}
     & = \left\lbrace \begin{array}{lcl}
           f_B(r(x))f_C(s(b))), & & \mbox{if} \,\, b \in B, c \in C  \,\, \mbox{and} \,\, x=bc  \\
                             0, & & \mbox{otherwise.}
           \end{array}\right.\\
       & \stackrel{s(b)=r(c)}{=} \left\lbrace \begin{array}{lcl}
                                 f_B(r(x))f_C(r(c)), & & \mbox{if} \,\, b \in B, c \in C  \,\, \mbox{and} \,\, x=bc  \\
                                                  0, & & \mbox{otherwise.}
                                 \end{array}\right.\\
       & \stackrel{r(x)=r(b)}{=} \left\lbrace \begin{array}{lcl}
                                 f_B(r(b))f_C(r(c)), & & \mbox{if} \,\, b \in B, c \in C  \,\, \mbox{and} \,\, x=bc  \\
                                                  0, & & \mbox{otherwise.}
                                 \end{array}\right.\\                                
       & = \sum_{\begin{array}{c}
                 s(b)=r(c) \\
                 x=bc
                 \end{array}} \psi(f_B\delta_B)(b)\psi(f_C \delta_C)(c)\\
       & = \psi(f_B \delta_B)\psi(f_C \delta_C)(x).
\end{align*}

In order to prove the surjectivity of $\psi$ let $f \in A_R(\G).$ We have seen that $f$ can be written as 
  $$ \sum_{i=1}^n r_i 1_{B_i},$$
where $n \in \N,$ $r_i \in R$, and $B_i$ are pairwise disjoint, for all $i=1, \cdots, n$. For each $i,$  we define  $f_{B_i}$ by
$$
f_{B_i}(x): = \left\lbrace \begin{array}{lcl}
                  r_i, & & \mbox{if} \,\, x \in r(B_i) \\
                    0, & & \mbox{if} \,\, x \in \Go \setminus r(B_i).
                \end{array}\right. 
$$
Clearly  $f_{B_i}$  belongs to $D_{B_i}.$ For $y \in \G$ we have that
\begin{align*}
\psi(f_{B_i}\delta_{B_i})(y) & = \left\lbrace \begin{array}{lcl}
                                 f_{B_i}(r(y)), & & \mbox{if} \,\, y \in B_i \\
                                             0, & & \mbox{if} \,\, y \notin B_i
                                 \end{array}\right. \\
                             & = \left\lbrace \begin{array}{lcl}
                                 r_i, & & \mbox{if} \,\, y \in B_i \\
                                 0, & & \mbox{if} \,\, y \notin B_i,
                                 \end{array}\right.
\end{align*}
that is,  $\psi(f_{B_i}\delta_{B_i})= r_i1_{B_i}.$ Taking $F= \sum_{i=1}^n f_{B_i}\delta_{B_i} \in \mathcal{L}$ we obtain that
  $$ \psi(F)= \sum_{i=1}^n r_i 1_{B_i} = f,$$ 
and thus $\psi$ is surjective.

Next we will show that $\psi(\mathcal{I})=\{0\},$  where $\mathcal{I}$ is the ideal of $\mathcal{L}$  generated by the set
 {\small $$\{f\delta_B-f\delta_A \, : \, B \leq A \,\,\, \mbox{and} \,\,\, f \in D_B \}=  \{f\delta_B-f\delta_A \, : \, B \subseteq A \,\,\, \mbox{and} \,\,\, f \in D_B  \subseteq D_A\}.$$}
Since $\psi$ is a homomorphism it is enough to show that $\psi$  is zero in the generators of $\mathcal{I}.$ Let $x \in \G.$ 

\begin{itemize}
\item If $x \notin A,$ then $x \notin B$ and $$ \psi(f\delta_B)(x)-\psi(f\delta_A)(x)=0. $$

\item If $x \in A \setminus B,$ then $r(x) \in r(A)\setminus r(B)$ ($r$ is injective in $A$).  Thus $$ \psi(f\delta_B)(x)-\psi(f\delta_A)(x)= \psi(f \delta_A)(x)=f(r(x))=0,$$ because $f \in D_B $ and $r(x) \notin r(B).$

\item If $x \in B,$ then  $x \in A$ and $$ \psi(f\delta_B)(x)-\psi(f\delta_A)(x)=f(r(x))-f(r(x))=0.$$
\end{itemize}
This proves that $\psi$ vanishes on $\mathcal{I}$.

We can now define a map $\tilde{\psi}$ of the quotient $\mathcal{L}/\mathcal{I}=\Lgo$ on $A_R(\G).$ Given $\overline{f_{B_i} \delta_{B_i}} \in \Lgo$ let
  $$ \tilde{\psi}\left(\sum_{i=1}^{n}\overline{f_{B_i} \delta_{B_i}}\right):= \psi\left(\sum_{i=1}^n f_{B_i}\delta_{B_i}\right).$$

To check that $\tilde{\psi}$ is well defined suppose that $ \sum_{i=1}^n \overline{f_{B_i}\delta_{B_i}}= \sum_{j=1}^m \overline{f_{A_j}\delta_{A_i}}.$ Then $ F:= \sum_{i=1}^n f_{B_i}\delta_{B_i}- \sum_{j=1}^m f_{A_j}\delta_{A_i} \,\, \in \,\, \mathcal{I}$ and hence $\psi(F)=0.$ Therefore
$$\tilde{\psi} \left(\sum_{i=1}^n \overline{f_{B_i}\delta_{B_i}}\right)= \psi\left(\sum_{i=1}^n f_{B_i}\delta_{B_i}\right)= \psi\left(\sum_{j=1}^m f_{A_j}\delta_{A_i}\right)= \tilde{\psi}\left(\sum_{j=1}^m \overline{f_{A_j}\delta_{A_i}}\right).$$

Clearly $\tilde{\psi}$ is a surjective homomorphism. In order to prove that $\tilde{\psi}$ is an isomorphism it suffices to verify that $\tilde{\psi}$ admits a left inverse. For this, consider the map $\varphi : A_R(\G) \rightarrow \Lgo$ defined as follows: Given $f \in A_R(\G)$, write it (uniquely) as
  $$f = \sum_{j=1}^n b_j 1_{B_j},$$ 
where $n \in N$ and $B_j$ are pairwise disjoint compact bisections of $\G$ for all $j=1, \cdots, n.$ Define 
  $$\varphi(f) = \varphi\left(\sum_{i=1}^n b_j 1_{B_j}\right) :=\sum_{j=1}^n \overline{b_j 1_{r(B_j)}\delta_{B_j}}. $$

We need to show that $\varphi \circ \tilde{\psi}$ is indeed the identity map of $\Lgo.$ 

Notice that, since each $f_B  \in D_B$ can be written as
  $$f_B= \sum_{k=1}^m c_k 1_{B_k},$$
where $m \in N, c_k \in R$ and $B_k$ are pairwise disjoint compact-open subset of $r(B)$ such that $\cup_{k=1}^m B_k \subseteq r(B),$ we have that
\begin{align*}
\psi(f_B \delta_B)(x) & = \left\lbrace \begin{array}{lcl}
                          \sum_{k=1}^m c_k 1_{B_k}(r(x)),  & & \mbox{if} \,\, x \in B \\
                          0,                               & & \mbox{if} \,\, x \notin B \\
                          \end{array}\right. \\
                      & = \left\lbrace \begin{array}{lcl}
                          c_k,  & & \mbox{if} \,\, x \in B \cap r^{-1}(B_k) \\
                          0,     & & \mbox{otherwise,}
                          \end{array}\right. \\
\end{align*} 
and, since the subsets $B \cap r^{-1}(B_k)$ are pairwise disjoint compact-open, we have that 
  $$ \psi(f_B \delta_B)= \sum_{k=1}^m c_k 1_{B \cap r^{-1}(B_k)}. $$ 
Thus,
\begin{align*}
\varphi\circ\tilde{\psi}(\overline{f_B \delta_B}) = \varphi \circ \psi (f_B \delta_B)  
                                        & = \varphi\left(\sum_{k=1}^m c_k 1_{B \cap r^{-1}(B_k)} \right) \\
                                        & = \sum_{k=1}^m \overline{ c_k 1_{r(B \cap r^{-1}(B_k))}\delta_{B \cap r^{-1}(B_k)}}\\
                                        & \stackrel{r(B \cap r^{-1}(B_k))=B_k}{=} \sum_{k=1}^m \overline{ c_k 1_{B_k} \delta_{B \cap r^{-1}(B_k)}  }\\
                                        & \stackrel{B \cap r^{-1}(B_k) \subseteq B}{=}\sum_{k=1}^m \overline{ c_k 1_{B_k} \delta_{B}  }\\
                                        & = \overline{\sum_{k=1}^m  (c_k 1_{B_k} \delta_{B})  }\\
                                        & = \overline{(\sum_{k=1}^m c_k 1_{B_k}) \delta_{B} }\\
                                        & = \overline{f_B \delta_{B} }.\\
\end{align*}

Notice that as $\tilde{\psi}$ is additive, if we prove that $\varphi$ is also additive, then we have that
 $$\varphi\circ\tilde{\psi}(\sum f_B \delta_B) = \sum( \varphi\circ\tilde{\psi} (f_B \delta_B))=\sum f_B \delta_B.$$
and so we conclude that $\tilde{\psi}$ is an isomorphism as desired. 

So it remains to prove that $\varphi$ is additive. Let $f=\sum_{i=1}^n r_i 1_{A_i}$ and $g=\sum_{j=1}^m s_j 1_{B_j}$ be functions in $A_R(\G)$, where $A_i$'s and $B_j$'s are pairwise disjoint compact bisections. We want to write $f+g$ also as a linear combination of characteristic functions of compact bisection that are disjoint.

To better understand what we will do next, consider that in Figure \ref{Rotulo} the rectangular regions are the bisections $A_i$'s, and the circular regions are the bisections $B_j$'s. Due to the disjunction of the bisections $A_i$'s and the $B_j$'s, we have only intersections of rectangular regions with circular ones (i.e., there are no intersections between circles and circles or between rectangles and rectangles).

\begin{figure}[H]
\centering
\includegraphics[scale=0.3]{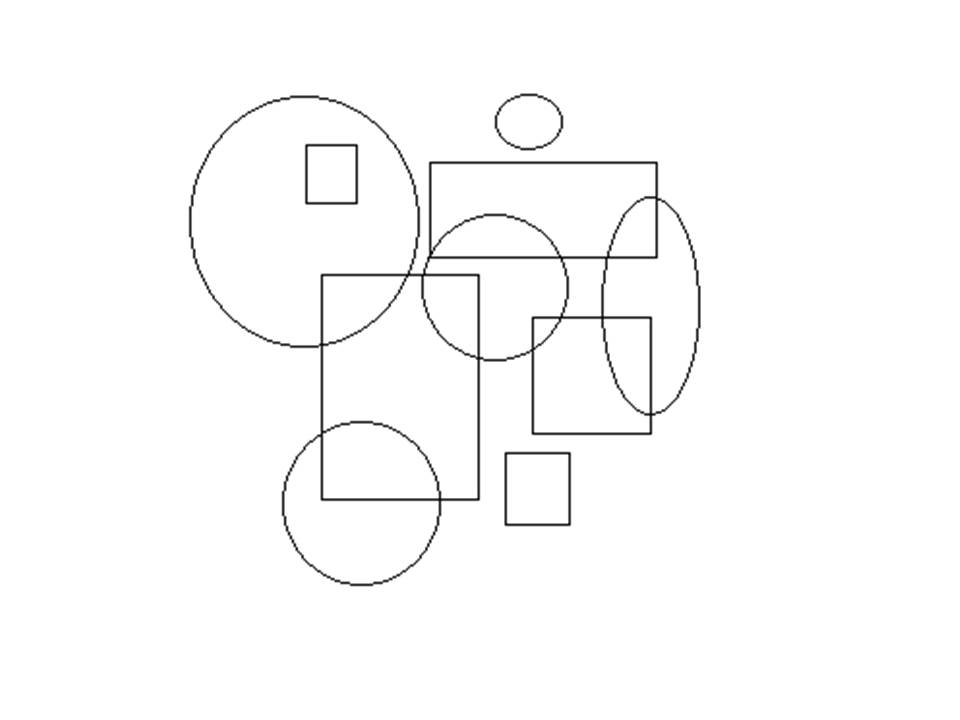}
\caption{}
\label{Rotulo}
\end{figure}

With these intersections we obtain three new types of regions, which are formed in the following ways:

$$ A_i\setminus(\cup_{j=1}^{m}B_j), \,\,\,\,\,\, B_j\setminus(\cup_{i=1}^{n}A_i),\,\,\,\, \mbox{and} \,\,\,\,\,  A_i \cap B_j, $$

where $i=1, \cdots, n$ and $j=1,\cdots, m.$ Note that some of these sets may be empty.

We have that $A_i \setminus (\cup_{j=1}^{m}B_j) = A_i \cap (\cup_{j=1}^{m}B_j)^c$ is a clopen subset of $A_i.$ Since $\G$ is Hausdorff, this set is compact and hence a compact bisection of $\G.$  Similarly  $B_j \setminus (\cup_{i=1}^{n}A_i)$ is a compact bisection of $\G.$ It is immediate that $A_i\cap B_j \subseteq A_i$ is open-compact and so it is also a compact bisection.

Let $x \in \G.$   
\begin{itemize}
\item If $x \in A_i \setminus (\cup_{j=1}^{m}B_j)$ then $f(x)=r_i, $ \vspace{0.1cm}
\item If $x \in B_j \setminus (\cup_{i=1}^{n}A_i)$ then $f(x)=s_j.$ \vspace{0.1cm}
\item If $x \in A_i \cap B_j$ then $(f+g)(x)= r_i+s_j.$
\end{itemize}

We want to rewrite $f+g$ in the form $\sum_k t_k1_{C_k}.$ To do so, let:

\begin{itemize}
\item $C_i := A_i \setminus (\cup_{j=1}^{m}B_j)$ \,\, and \,\, $t_i:=r_i,$ \,\, for all \,\, $i=1, \cdots, n;$ \vspace{0.1cm}
\item $C_{n+j}:= B_j \setminus (\cup_{i=1}^{n}A_i) $ \,\, and \,\, $t_{n+j}:=s_j,$ \,\, for all \,\, $j=1, \cdots ,m;$ \vspace{0.1cm}
\item $C_{n+m+(j-1)n+i} := A_i \cap B_j$ \,\, and \,\, $t_{n+m+(j-1)n+i}:=r_i+s_j,$ \,\, for all \,\, $i=1, \cdots, n$ and $j=1, \cdots, m.$
\end{itemize}

With these definitions we have that
$$ f+g=\sum_{k=1}^{n+m+m\cdot n} t_k 1_{C_k}. $$

Note that if $C_k$ is empty then $1_{C_k} $ and $1_{r(C_k)}$ are null functions and do not interfere with our calculation below. Also note that:

\begin{itemize}
\item $\overline{r_i 1_{r(A_i\setminus \cup B_j)} \delta_{A_i \setminus \cup B_j}} = \overline{r_i 1_{r(A_i\setminus \cup B_j)} \delta_{A_i }}$ because $A_i \setminus \cup B_j  \subset A_i;$ \vspace{0.15cm}

\item $\overline{s_j 1_{r(B_j\setminus \cup A_i)} \delta_{B_j \setminus \cup A_i}} = \overline{s_j 1_{r(B_j \setminus \cup A_i)} \delta_{B_j}}$ because $B_j \setminus \cup A_i  \subset B_j;$ \vspace{0.15cm}

\item $\overline{r_i 1_{r(A_i \cap B_j)} \delta_{A_i \cap B_j}} = \overline{r_i 1_{r(A_i \cap B_j)} \delta_{A_i }}$ because $A_i \cap B_j \subset A_i;$ \vspace{0.15cm}

\item $\overline{s_j 1_{r(A_i \cap B_j)} \delta_{A_i \cap B_j}} = \overline{s_j 1_{r(A_i \cap B_j)} \delta_{B_j }}$ because $A_i \cap B_j \subset B_j.$
\end{itemize}
  
We now obtain that $\varphi(f+g)=\varphi(f)+\varphi(g)$ from the following computation:

\begin{align*}
\varphi(f+g) & = \sum_k^{n+m+m \cdot n} \overline {t_k 1_{r(C_k)}\delta_{C_k}} \\
             & = \sum_{i=1}^n \overline{r_i 1_{r(A_i \setminus \cup B_j)} \delta_{A_i \setminus \cup B_j}}
               + \sum_{j=1}^m \overline{s_j 1_{r(B_j\setminus \cup A_i)} \delta_{B_j \setminus \cup A_i}} \\
             & + \sum_{i=1}^n \sum_{j=1}^m \overline{(r_i + s_j) 1_{r(A_i \cap B_j)} \delta_{A_i \cap B_j}} \\
             & = \sum_{i=1}^n \overline{r_i 1_{r(A_i \setminus \cup B_j)} \delta_{A_i \setminus \cup B_j}}
               + \sum_{j=1}^m \overline{s_j 1_{r(B_j\setminus \cup A_i)} \delta_{B_j \setminus \cup A_i}} \\
             & + \sum_{i=1}^n \sum_{j=1}^m \overline{r_i  1_{r(A_i \cap B_j)} \delta_{A_i \cap B_j}}
               + \sum_{i=1}^n \sum_{j=1}^m \overline{s_j 1_{r(A_i \cap B_j)} \delta_{A_i \cap B_j}} \\
             & = \sum_{i=1}^n \overline{r_i 1_{r(A_i \setminus \cup B_j)} \delta_{A_i}}
               + \sum_{j=1}^m \overline{s_j 1_{r(B_j\setminus \cup A_i)} \delta_{B_j }} \\
             & + \sum_{i=1}^n \sum_{j=1}^m \overline{r_i  1_{r(A_i \cap B_j)} \delta_{A_i}}
               + \sum_{i=1}^n \sum_{j=1}^m \overline{s_j 1_{r(A_i \cap B_j)} \delta_{B_j}} \\  
             & = \sum_{i=1}^n \overline{[r_i( 1_{r(A_i\setminus \cup B_j)}+ \sum_{j=1}^m 1_{r(A_i \cap B_j)})]\delta_{A_i}}
               + \sum_{j=1}^m \overline{[s_j (1_{r(B_j\setminus \cup A_i)}+ \sum_{i=1}^n 1_{r(A_i \cap B_j)})] \delta_{B_j}}\\
             & = \sum_{i=1}^n \overline{r_i 1_{r(A_i)}\delta_{A_i}} 
               + \sum_{j=1}^m \overline{s_j 1_{r(B_j)} \delta_{B_j}}  \\
             & = \varphi(f) + \varphi (g). 
\end{align*}

\end{proof}

\noindent V. M. Beuter, Pós-Graduação em Matemática Pura e Aplicada, Universidade Federal de Santa Catarina, Florianópolis, 88040-900, Brasil and Departamento de Matemática, Universidade do Estado de Santa Catarina, Joinville, 89219-710, Brasil.\\ 
E-mail: vivibeuter@gmail.com \\

\noindent D. Gonçalves, Departamento de Matemática, Universidade Federal de Santa Catarina, Florianópolis, 88040-900, Brasil.\\
E-mail: daemig@gmail.com

\end{document}